\documentclass[11pt]{article}

\usepackage[utf8]{inputenc}
\usepackage[margin=0.9in]{geometry}
\usepackage{authblk}
\usepackage{amsmath,amssymb,amsthm,mathtools}
\usepackage{aliascnt}
\usepackage{microtype}
\usepackage{enumitem}
\usepackage[hidelinks]{hyperref}
\usepackage[nameinlink,capitalize,noabbrev]{cleveref}

\hypersetup{
  pdftitle={A Short Proof of the Hypergraph Moore Bound},
  pdfauthor={}
}

\newtheorem{theorem}{Theorem}[section]
\newaliascnt{proposition}{theorem}
\newtheorem{proposition}[proposition]{Proposition}
\aliascntresetthe{proposition}
\newaliascnt{lemma}{theorem}
\newtheorem{lemma}[lemma]{Lemma}
\aliascntresetthe{lemma}
\newaliascnt{corollary}{theorem}

\aliascntresetthe{corollary}
\newaliascnt{definition}{theorem}
\theoremstyle{definition}
\newtheorem{definition}[definition]{Definition}
\aliascntresetthe{definition}
\theoremstyle{remark}
\newaliascnt{remark}{theorem}

\aliascntresetthe{remark}

\crefname{theorem}{theorem}{theorems}
\crefname{proposition}{proposition}{propositions}
\crefname{lemma}{lemma}{lemmas}
\crefname{corollary}{corollary}{corollaries}
\crefname{definition}{definition}{definitions}
\crefname{remark}{remark}{remarks}

\newcommand{\F}{\mathbb F}
\newcommand{\1}{\mathbf 1}

\newcommand{\cB}{\mathcal B}
\newcommand{\cH}{\mathcal H}

\newcommand{\cC}{\mathcal C}

\newcommand{\cT}{\mathcal T}
\newcommand{\Lop}{\mathcal L}
\newcommand{\Om}{\Omega}
\newcommand{\sd}{\mathbin{\triangle}}
\newcommand{\bigsd}{\mathop{\bigtriangleup}}
\newcommand{\bd}{\bar d}
\newcommand{\Rh}[1]{\widehat R_{#1}}
\newcommand{\norm}[1]{\left\lVert #1\right\rVert}
\DeclareMathOperator{\Tr}{Tr}
\DeclareMathOperator{\diag}{diag}

\numberwithin{equation}{section}
\allowdisplaybreaks
\setlist[enumerate]{leftmargin=2.1em,itemsep=0.2em,topsep=0.3em}
\title{A Spectral Proof of the Hypergraph Moore Bound}
\author[1,2]{Alexander Schmidhuber}
\author[2]{Matthew B. Hastings}

\affil[1]{Center for Theoretical Physics, MIT}
\affil[2]{Microsoft Station Q}
\date{}

\begin{document}
\maketitle

\vspace{-2em}

\begin{abstract}
A nonempty subfamily of a $k$-uniform hypergraph is an \emph{even cover}
if every vertex lies in an even number of its hyperedges; for $k=2$ these
are edge-disjoint unions of cycles, so the minimum size of an even cover
is the natural hypergraph analogue of girth.  We prove Feige's 2008 conjecture on
the hypergraph Moore bound: there are absolute constants $A$ and $C$ (independent of $k$) such
that for every $k\ge3$ and every $1\le\ell\le n$, any $k$-uniform
hypergraph on $n$ vertices with more than $C\,n^{k/2}/\ell^{k/2-1}$
hyperedges contains an even cover of size at most $A\,\ell\log(en/\ell)$.
Our proof is based on sharp spectral bounds for Kikuchi matrices, which
we expect to be of independent interest; we apply them to the refutation of random constraint satisfaction problems  in a companion paper~\cite{SH}.
\end{abstract}

\section{Introduction}

The Moore bound in its irregular form, due to Alon, Hoory, and
Linial~\cite{AHL}, states that large average degree forces a short cycle in a
graph.  Feige~\cite{Feige} proposed a
hypergraph analogue.  Let $\cH\subseteq\binom{[n]}k$ be a $k$-uniform
hypergraph.  A nonempty subfamily $\cC\subseteq\cH$ is an \emph{even cover}
if every vertex lies in an even number of members of $\cC$; writing $\sd$ for the symmetric difference, this says
\[
        \bigsd_{F\in\cC}F=\varnothing.
\]
For $k=2$ an even cover is precisely an edge-disjoint union of cycles, so the
minimum size $g_{\rm ev}(\cH)$ of an even cover (declared $+\infty$ if none
exists) is the natural hypergraph analogue of girth.  Short even covers are
short linear dependencies among sparse binary vectors, which links
$g_{\rm ev}$ to parity-check matrices of low-density codes and to refutation
certificates for random constraint satisfaction
problems~\cite{NaorVerstraete,Feige,FKO,GKM}.

Feige~\cite{Feige} conjectured that
$C_k n(n/\ell)^{k/2-1}$ hyperedges force an even cover of size
$O_k(\ell\log n)$, giving a hypergraph analogue of the irregular
Moore bound and extending an earlier theorem of Naor and
Verstra\"ete~\cite{NaorVerstraete} from the high-density endpoint
$\ell=\Theta(1)$ for even $k$ to the full range of $\ell$. Guruswami,
Kothari, and Manohar~\cite{GKM} obtained the first bound valid in the whole
range of $\ell$, at a density a factor $\mathrm{polylog}(n)$ above the conjectured
threshold \cite{GKM,HKM}; their
 argument introduced Kikuchi matrices, a hierarchy defined by
Wein, El Alaoui, and Moore~\cite{WEM} (and in a different language by Hastings~\cite{Hastings2019}), into this problem.  Hsieh, Kothari,
and Mohanty~\cite{HKM} reduced the loss to a single logarithm, showing
that $\log n \cdot n^{k/2}/\ell^{k/2-1}$ hyperedges, up to a constant factor,
force an even cover of size $O(\ell\log n)$. For odd uniformities, Hsieh et al.~\cite{HKMMS} lowered the
loss further, to $(\log n)^{1/(k+1)}$ in the range $\ell\le n/\log^2n$.
We prove the conjecture in full, without any logarithmic losses. Our proof yields absolute constants (independent of $k$) and a cover size $O(\ell \log(en/\ell))$ that is slightly 
sharper than the conjectured $O(\ell \log n)$.

\begin{theorem}[Hypergraph Moore bound]\label{thm:main}
There exist absolute constants $A,C>0$ such that for every $k\ge3$,
every $1\le\ell\le n$, and every $k$-uniform hypergraph $\cH\subseteq\binom{[n]}k$, at
least one of
\[
        g_{\rm ev}(\cH)\le A\,\ell\log\frac{en}{\ell},
        \qquad
        |\cH|\le C\,\frac{n^{k/2}}{\ell^{k/2-1}}
\]
holds.
\end{theorem}


All logarithms are natural.  The hypergraphs we study are \emph{simple}: $\cH$ is a \emph{set}
of $k$-element subsets of $[n]$.  Throughout, $A$ and $C$ denote positive
absolute constants (independent of all parameters, including $k$); we make no attempt to optimize them.

\subsection{Overview of the proof}\label{sec:overview}

\paragraph{The Kikuchi graph.}
For simplicity, let us first consider the case of $k=2r$ even. We work on the
 $\ell$-slice $\Om_\ell=\binom{[n]}\ell$ of the hypercube, and consider the
\emph{Kikuchi graph}\footnote{The name originates from~\cite{WEM}, motivated by the hierarchy of
Kikuchi approximations to free energy in statistical physics~\cite{Kikuchi1951}.  As the authors
of~\cite{WEM} note, their derivation expands about an inconsistent paramagnetic
point, so the connection between Kikuchi matrices and the Hessian of Kikuchi
free energy holds only to first order.  A more direct physical interpretation
comes from mean-field theory; see~\cite{Hastings2019} and Appendix~B
of~\cite{Schmidhuber2024}.} $G$ in
which two rows are adjacent when they differ by a hyperedge:
$S\sim S'$ whenever $S\sd S'\in\cH$, as in~\cite{WEM,GKM,HKM}.
A hyperedge forms an edge exactly at the rows containing half of it, so a
simple count shows that the average degree $\bd$ of $G$ is of order
$|\cH|(\ell/n)^{k/2}$.  The conjectured density
$n^{k/2}/\ell^{k/2-1}$ is precisely where $\bd$ reaches the row size
$\ell$.  \Cref{thm:main} thereby reduces to one statement:
\begin{center}
\emph{if $\cH$ has no even cover of size $O(\ell\log(en/\ell))$, then
$\bd=O(\ell)$.}
\end{center}

\paragraph{Two spectral bounds.}
We prove \Cref{thm:main} by exhibiting a contradiction between two spectral
bounds on a reweighted version of the Kikuchi matrix: a trivial lower bound, valid for every nonempty $\cH$, and an
upper bound derived from the girth hypothesis (that is, the assumption $g_{\rm ev}(\cH) > A\,\ell\log\frac{en}{\ell}$) when $\bd$ is too large.
Some reweighting is forced on us, since neither the adjacency matrix
$A_G$ of $G$ nor its fully degree-normalized version can carry such a
contradiction.  On $A_G$ itself there is the trivial lower bound
$\norm{A_G}\ge\bd$, but no useful upper bound: the norm of $A_G$ is governed
by the densest neighborhoods rather than by the average (one row of
degree $\Delta$ already forces $\norm{A_G}\ge\sqrt\Delta$) and the girth
hypothesis does nothing to prevent such irregularity.  The naive normalization by
degrees also overshoots: $D^{-1/2}A_GD^{-1/2}$, with $D$ the
diagonal degree matrix of $G$, has top eigenvalue exactly $1$ for \emph{every}
graph without isolated vertices, so an upper bound on its norm carries no
information.  We work instead with the normalization introduced in~\cite{HKM} for the same task:
\[
 K=\Gamma^{-1/2}A_G\,\Gamma^{-1/2},
 \qquad
 \Gamma=D+\bd\,I .
\]
Then a simple lower bound via the vector $\Gamma^{1/2}\1$ gives
$\norm K\ge\1^{\mathsf T}A_G\1/\1^{\mathsf T}\Gamma\1=\tfrac12$ for every
nonempty $\cH$.  Our
goal is now the complementary estimate: \emph{if $\cH$ has no short even
cover, then $\norm{K}<\frac{1}{2}$ unless $\bd=O(\ell)$.}

\paragraph{The trace method and the parity lift.}
The upper bound is proved by the trace method~\cite{FurediKomlos}:
For any integer $s \geq 0$, $\norm K^{2s}\le\Tr K^{2s}$, and the trace is a weighted count of closed
walks of length $2s$ in $G$.  Our main contribution is a technique for accurately counting these walks.

This count must somewhere use the high-girth
hypothesis.  Each step
of a walk traverses a hyperedge; call a closed walk \emph{even} if it
traverses every hyperedge an even number of times.  Even closed
walks (for instance, walks that repeatedly backtrack) arise in every
Kikuchi graph, and contribute
a significant part of the spectrum.  The girth hypothesis constrains exactly the
other, ``odd'' walks: around a closed walk the hyperedges XOR to
$\varnothing$, so those used an even number of times drop out and those
used an odd number of times XOR to $\varnothing$ among themselves; if
there is at least one of them, they form an even cover of size at most $2s$. At our working scale $s = O(\log(|\Om_\ell|)/k)$ this is precisely forbidden by
the high-girth hypothesis.  The upper bound therefore amounts to counting
even walks accurately, and separately from the rest.

Our main technical ingredient is to lift the Kikuchi graph to a graph with memory:  A lifted state is a pair $(S,M)$ of a row $S$ and
a \emph{memory} $M$, the set of hyperedges used an odd number of times
so far.  A step along a hyperedge $F$ replaces $S$ by $S\sd F$ and
toggles $F$ into or out of the memory: a first use of $F$ enters it, a
second use removes it, a third enters it again.  A walk starting from empty memory
$(S,\varnothing)$ returns to
$(S,\varnothing)$ exactly when its projection to $G$ is an even closed
walk at $S$.  When starting at empty memory, the girth hypothesis yields two useful properties:
every closed walk of length $2s$ ends with empty memory, and along the
way its memory never exceeds $s$ hyperedges, since each hyperedge in the
memory must be used at least twice.  Consequently we may express walks on $K$ by walks on the lifted graph, with memory truncated at size $s$:
\begin{equation}
 2^{-2s}\le\norm K^{2s}\le\Tr K^{2s}
 =\sum_{S}\bigl\langle(S,\varnothing)\bigr|\,\Lop^{2s}\,
 \bigl|(S,\varnothing)\bigr\rangle
 \le N\norm \Lop^{2s}.
 \label{eq:trace_main}
\end{equation}
Here $\Lop$ is the lifted walk operator, with the same step weights as $K$,
and $N=|\Om_\ell|$.

\paragraph{Orienting the lifted graph.}
To obtain an accurate upper bound on $\norm{K}$, everything now hinges on bounding $\norm\Lop$. We do so by \emph{orienting} the lifted edges in one of the two directions. It is easy to establish that if the lifted graph admits an orientation such that every state has \emph{incoming} weight at
most $\kappa$, then the walk operator $\Lop$ satisfies $\norm\Lop\le2\sqrt{\kappa/\bd}$ (the denominator arises from the normalization by $\Gamma$).
Combined with equation \eqref{eq:trace_main} and the high-girth hypothesis, this forces
$\bd\le16\,\kappa\,N^{1/s}$.  The theorem is thereby reduced to a
combinatorial task: orient the lifted graph so that little weight points
\emph{into} any one state.  Here there is a revealing warm-up: one may choose the trivial orientation which orients
every edge toward the larger memory. This already bounds the incoming weight by
$|M|\le s$ with no combinatorial input at all, and recovers the Moore bound
with a single logarithmic loss as in \cite{HKM}.  The sharp bound is purely a
better orientation.  A state $(S,M)$ reachable from empty memory involves only a small 
total \emph{support} (the union of $S$ and all hyperedges in $M$), of size at most $\ell+ks$.  The girth hypothesis makes the
hyperedges confined to such a window linearly independent over $\F_2$;
independent vectors satisfy Hall's marriage condition, so the hyperedges
in a window admit a system of distinct representatives.  Orienting each
support-preserving edge toward the endpoint whose row contains the
representative of its hyperedge, and each support-changing edge toward
the larger support, caps the incoming weight at $2\ell$---at most
$\ell$ from each rule---so $\bd\le32\ell N^{1/s}$. 
To obtain the uniform statement of \Cref{thm:main} we choose the walk length
$s\asymp\ell\log(en/\ell)/k$: the factor $k$ keeps every window and closed
walk inside radius $O(\ell\log(en/\ell))$, at the price of an average degree larger by $N^{1/s}=e^{O(k)}$. Fortunately, this factor  is absorbed by the conversion of $|\cH|$ into $\bd$.  This proves the even case.


\paragraph{Odd uniformity: pairing hyperedges.}
Evenness of $k$ entered the argument in exactly one place: a row can
contain half of a hyperedge, so a single hyperedge preserves the $\ell$-slice.
For odd $k$ this fails. A natural solution is to consider walks that traverse hyperedges two at a
time, and then apply our even-$k$ techniques.  We make this rigorous with
a variant of the high-codegree decompositions of the odd-uniformity
Kikuchi arguments~\cite{GKM,HKM,HKMMS} by constructing a maximal packing of
families of hyperedges that share a common center.  

A genuinely new difficulty remains: a fixed row $S$ and hyperedge $F$ can lie on many
moves, one for each eligible partner hyperedge $F'$, while the orientation argument above,
which pays its bounds per row--hyperedge incidence, cannot afford
unboundedly many moves through a single incidence.  The arguments of~\cite{HKM,HKMMS} prune: \cite{HKM} retains only moves isolated at both endpoints, and \cite{HKMMS} only moves whose endpoints carry boundedly many moves of the same colour, inside a carefully restricted subgraph of an edge-colored Kikuchi graph; both must then show that most moves survive. We instead keep every move and assign it the weight
$1/\delta_e$, where $\delta_e$ is one plus the total excess ``contention''
of its four possible row--hyperedge incidences. The
same core technique as before---lifting the graph with a memory that keeps track of the hyperedges used an odd number of times so far, and choosing an orientation on that graph using Hall's marriage theorem---then proves  the case of odd uniformity.

\paragraph{Concurrent work.}
While preparing this manuscript for publication, an independent preprint by Bandeira, Kunisky, Nizi\'c-Nikolac, Pesenti, and Wang appeared on the arXiv, proving
the conjecture first for even $k$, and in a subsequent revision for all $k$~\cite{BKNPW}. Their elegant argument proceeds
through minimum-degree cores, neighborhood growth, and
polynomial interpolation. Our results were obtained independently and prior to the appearance of~\cite{BKNPW}, and are instead   based on counting trace walks on the Kikuchi graph.
Compared to~\cite{BKNPW}, this yields $k$-independent constants in \Cref{thm:main} and also finds applications to other problems such as those in our companion paper~\cite{SH}.

\paragraph{On the use of AI models.}
All proof ideas were conceived and developed by the authors. We have used Large Language Models  to assist with writing and proofreading. All mistakes are ours.

\section{Even uniformity}\label{sec:even}

Throughout this section and the next, we choose the hierarchy level $\ell$ to lie in the
\emph{central range}
\begin{equation}\label{eq:central}
        2k\le\ell,
        \qquad
        n\ge100\,k\ell,
\end{equation}
which is the interesting regime in the context of the hypergraph Moore bound. 
The boundary values of $\ell$ are treated in \cref{sec:completion}:
above the central range a short syndrome count suffices, and below it,
it is enough to run the present argument at the level $2k$ or the argument of \cref{sec:odd}.

Let $k=2r\ge4$ be even.  Set
\begin{equation}\label{eq:even-moment}
        s:=\Bigl\lceil\frac{3\ell}k\log\frac{en}\ell\Bigr\rceil,
        \qquad
        N:=\binom n\ell,
        \qquad\text{so that}\qquad
        N^{1/s}\le e^{k/3},
\end{equation}
since $\log N\le\ell\log(en/\ell)$.

\begin{theorem}\label{thm:even-main}
Assume \eqref{eq:central}.  If $\cH$ has no even cover of size at most
$5\ell\log(en/\ell)$, then
\[
        |\cH|\le45\,\frac{n^{k/2}}{\ell^{k/2-1}}.
\]
\end{theorem}

\subsection{The Kikuchi graph and normalization}

Let $G$ be the graph on the slice $\Om_\ell:=\binom{[n]}\ell$ in which
a hyperedge $F$ joins $S$ to $S\sd F$ whenever $|S\cap F|=r$.  Write
$A_G$ for its adjacency matrix, $d(S)$ for its degree, and
$\bd:=N^{-1}\sum_Sd(S)$.  Counting the rows adjacent through a fixed
$F$ gives
\begin{equation}\label{eq:even-degree-exact}
 \bd=|\cH|\binom{2r}{r}
       \frac{\binom{n-2r}{\ell-r}}{\binom n\ell}.
\end{equation}

The slice ratio in \eqref{eq:even-degree-exact} is estimated by the
following bound, stated for general $q$ because the odd case will
reuse it.  
To obtain $k$-independent constants throughout, we accurately keep track of factors involving $k$ or $q$.    

\begin{lemma}\label{lem:ratio}
Let $1\le q\le k$.  Prescribing $q$ vertices inside a row and $q$
vertices outside it retains a definite fraction of the slice:
\[
 \frac{\binom{n-2q}{\ell-q}}{\binom n\ell}
 \ \ge\ \frac9{10}\Bigl(\frac2e\Bigr)^{q}
 \Bigl(\frac\ell n\Bigr)^{q}.
\]
\end{lemma}

\begin{proof}
Writing $(x)_j:=x(x-1)\cdots(x-j+1)$ for the falling factorial, one has
the identity
\[
 \Bigl(\frac n\ell\Bigr)^{q}
 \frac{\binom{n-2q}{\ell-q}}{\binom n\ell}
 =\frac{(\ell)_q}{\ell^q}\,
   \frac{(n-\ell)_q}{n^q}\,
   \frac{n^{2q}}{(n)_{2q}} .
\]
The third factor is at least $1$, and the second is at least
$(1-(\ell+q)/n)^q\ge1-2q\ell/n\ge1-\tfrac2{100}$ by
\eqref{eq:central}.  The first factor increases with $\ell$, and
$\ell\ge2k\ge2q$ by \eqref{eq:central}, so it suffices to bound it at
$\ell=2q$, where Stirling,
$(2q)!\ge\sqrt{4\pi q}\,(2q/e)^{2q}$ and $q!\le e\sqrt q\,(q/e)^q$,
gives
\[
 \frac{(2q)_q}{(2q)^q}
 =\frac{(2q)!}{q!\,(2q)^q}
 \ge\frac{2\sqrt\pi}{e}\Bigl(\frac2e\Bigr)^q
 \ge\Bigl(\frac2e\Bigr)^q. \qedhere
\]
\end{proof}

\begin{lemma}\label{lem:even-density}
The Kikuchi graph converts hyperedges into degree according to
\[
        \bd\ge\frac9{20\sqrt r}\Bigl(\frac8e\Bigr)^{r}
        \Bigl(\frac\ell n\Bigr)^{r}|\cH| .
\]
\end{lemma}

\begin{proof}
Combine \eqref{eq:even-degree-exact}, \cref{lem:ratio} with $q=r$, and
the standard bound $\binom{2r}{r}\ge4^r/(2\sqrt r)$, which follows by
induction from $(2r+1)^2\ge4r(r+1)$.
\end{proof}

Assume $\cH\ne\varnothing$, so $\bd>0$, and put
$\Gamma:=\diag(d(S))+\bd I$ and $K:=\Gamma^{-1/2}A_G\Gamma^{-1/2}$.
Testing on the vector $\Gamma^{1/2}\1$ gives the universal lower bound
$\norm K\ge\1^{\mathsf T}A_G\1/\1^{\mathsf T}\Gamma\1=\frac12$.

\subsection{The parity lift}

The truncated lift has states $(S,M)$ with $S\in\Om_\ell$,
$M\subseteq\cH$, and $|M|\le s$.  A step labeled by $F$ sends
$(S,M)\leftrightarrow(S\sd F,M\sd\{F\})$ with normalized weight
$(\Gamma(S)\Gamma(S\sd F))^{-1/2}$.  We restrict to the union of the
components containing a state $(S,\varnothing)$ with empty memory, and
denote the resulting symmetric operator by $\Lop$.

\begin{lemma}\label{lem:even-trace}
If $\cH$ has no even cover of size at most $2s$, then
\[
 \Tr K^{2s}
 =\sum_{S\in\Om_\ell}
 \bigl\langle(S,\varnothing)\bigr|\,\Lop^{2s}\,
 \bigl|(S,\varnothing)\bigr\rangle.
\]
\end{lemma}

\begin{proof}
Both sides are weighted counts of closed $2s$-step walks, and we
exhibit a weight-preserving bijection between the two families.  Since
$G$ is simple, a closed $2s$-step walk at $S\in\Om_\ell$ determines,
and is determined by, its sequence of labels; it therefore lifts to a
unique walk of the lifted graph started at $(S,\varnothing)$, of the
same weight, and conversely a closed lifted walk at $(S,\varnothing)$
projects to a closed base walk at $S$.  The two maps are mutually
inverse, so it remains only to check that the lift of a closed base
walk is again closed and is not discarded by the truncation.

Let $Q$ be the set of labels used an odd number of times.  Row closure
gives $\bigsd_{F\in Q}F=\varnothing$, so the cover hypothesis forces
$Q=\varnothing$: the lift ends with empty memory.  Consequently, at
time $j$ the memory has size at most $j$, and also at most $2s-j$,
because every hyperedge held in memory must be used again before the
walk closes; hence it never exceeds $s$.  The same bound holds
automatically along a closed lifted walk at $(S,\varnothing)$, so the
truncation discards nothing.
\end{proof}

\begin{lemma}\label{lem:orientation}
Suppose the lifted edges are oriented so that every state has in-degree
at most $\kappa$.  Then $\norm\Lop\le2\sqrt{\kappa/\bd}$.  The same holds for
nonnegative edge weights, with in-degree replaced by incoming weight
and with the row degree $d(S)$ weighted accordingly.
\end{lemma}

\begin{proof}
Write $S_x$ for the row of a lifted state $x$.  Since $\Lop$ is symmetric,
$\norm\Lop=\sup\{|\langle f,\Lop f\rangle|:\norm f=1\}$, and since the
entries of $\Lop$ are nonnegative, $|\langle f,\Lop f\rangle|\le\langle
|f|,\Lop|f|\rangle$; we may therefore assume $f\ge0$.  Each of the lifted graph's edges contributes to exactly
two entries of $\Lop$, and, summing over the edges with their given
orientation,
\[
 \langle f,\Lop f\rangle
 =\sum_{y\to x}\frac{2f(x)f(y)}{\sqrt{\Gamma(S_x)\Gamma(S_y)}} .
\]
For $t>0$ the arithmetic--geometric mean inequality bounds each summand by
\[
 \frac{2f(x)f(y)}{\sqrt{\Gamma(S_x)\Gamma(S_y)}}
 \le t\frac{f(x)^2}{\Gamma(S_x)}
      +t^{-1}\frac{f(y)^2}{\Gamma(S_y)} ,
\]
and collecting the two families of terms by their head and by their tail,
\[
 \langle f,\Lop f\rangle
 \le t\sum_{x}\frac{\deg^-(x)}{\Gamma(S_x)}f(x)^2
  +t^{-1}\sum_{y}\frac{\deg^+(y)}{\Gamma(S_y)}f(y)^2 ,
\]
where $\deg^\mp$ denote the in- and out-degrees of the orientation.  In the
first sum $\deg^-(x)\le\kappa$ and $\Gamma(S_x)\ge\bd$, so it is at most
$(\kappa/\bd)\norm f^2$; in the second, $\deg^+(y)$ is at most the full lifted
degree at $y$, which is at most $d(S_y)\le\Gamma(S_y)$, so it is at most
$\norm f^2$.  Hence $\norm\Lop\le t\kappa/\bd+t^{-1}$ for every $t>0$, and
$t=\sqrt{\bd/\kappa}$ gives $\norm\Lop\le2\sqrt{\kappa/\bd}$.  In the weighted case
each edge carries its weight $w_e$ through the same computation: in-degrees
become incoming weights, bounded by $\kappa$, and the out-weight at $y$ is
bounded by the weighted degree $d(S_y)$.
\end{proof}

\subsection{Small windows and the sharp orientation}

For a lifted state $x=(S,M)$, write $W(x):=S\cup\bigcup_{F\in M}F$;
since $|M|\le s$, we have $|W(x)|\le\ell+ks$.

\begin{lemma}\label{lem:even-window}
Suppose $\cH$ has no even cover of size at most $\ell+ks+1$.  Then for
every $W\subseteq[n]$ with $|W|\le\ell+ks$, the hyperedges of
$\cH[W]:=\{F\in\cH:F\subseteq W\}$ admit a system of distinct
representatives.
\end{lemma}

\begin{proof}
A dependence among the corresponding incidence vectors contains a
minimal dependence of size at most $|W|+1$, hence a forbidden even
cover.  Thus the vectors are independent.  Every subfamily is supported
on its union, so its size is at most the size of that union; this is
Hall's condition, and by Hall's marriage theorem implies the existence of a system of disting representatives.
\end{proof}

Along a lifted edge the two supports are nested: if $F$ lies in the
memory of $x$ but not in that of $y$, then $S_x=S_y\sd F$ and
$(S_y\sd F)\cup F=S_y\cup F$, so $W(x)=W(y)\cup F\supseteq W(y)$.  (This
is special to the even case; in \cref{sec:odd} a step may enter one
label and remove the other.)  Fix one
representative system $\rho_W$ for every window $W$ with $|W|\le\ell+ks$,
and orient a lifted edge labeled by $F$ as follows.  If its endpoint
supports agree, say $W(x)=W(y)=W$, orient toward the endpoint whose row
contains $\rho_W(F)$ (it cannot be contained in both).  If the supports differ, orient toward the larger
one.  In both cases we call the vertex responsible for the choice the
\emph{witness} of the edge: it is $\rho_W(F)$ when the supports agree,
and the smallest vertex of $W(x)\setminus W(y)$, with $x$ the head, when
they differ.

\begin{lemma}\label{lem:even-indegree}
Suppose $\cH$ has no even cover of size at most
$\max\{2s,\ \ell+ks+1\}$.  Then every lifted state has in-degree at
most $2\ell$ under the above orientation, and
\begin{equation}\label{eq:even-spectral}
        \norm K\ \le\ 2\sqrt{\frac{2\ell}{\bd}}\;N^{1/(2s)} .
\end{equation}
\end{lemma}

\begin{proof}
The hypothesis contains that of \cref{lem:even-window}, so the
representative systems $\rho_W$ exist and the orientation is defined.
Fix a state $x=(S,M)$.  Equal-support in-edges have distinct labels and
hence distinct witnesses in the $\ell$-set $S$, giving at most $\ell$ of
them.

For an unequal-support in-edge $e\colon y\to x$, let
$v\in W(x)\setminus W(y)$ be its witness and let $F$ be its label.  First, $v\in S$: otherwise $v$ lies in
some $F'\in M$; if $F'\ne F$ then $F'\in M_y$ and $v\in W(y)$, while if
$F'=F$ then, since the step moves exactly the vertices of $F$,
$v\notin S$ forces $v\in S_y\subseteq W(y)$.  Next, $v\in S\setminus
S_y\subseteq F$, and $F\in M$: had the step removed $F$ from the
memory, $F\in M_y$ would give $v\in W(y)$.  Finally, $F$ is the unique
member of $M$ containing $v$, because every other member of $M$ lies in
$M_y$.  Hence $v$ determines the label $F$ and thereby the tail
$y=(S\sd F,M\sd\{F\})$: two unequal-support in-edges with the same $v$
coincide.  This gives at most $\ell$ further in-edges.

Finally, \cref{lem:orientation} turns the in-degree bound into
$\norm\Lop\le2\sqrt{2\ell/\bd}$, so that, by \cref{lem:even-trace},
\begin{equation}\label{eq:even-chain}
 \norm K^{2s}\le\Tr K^{2s}
 =\sum_{S\in\Om_\ell}\bigl\langle(S,\varnothing)\bigr|\,\Lop^{2s}\,
 \bigl|(S,\varnothing)\bigr\rangle
 \le N\norm\Lop^{2s}
 \le N\Bigl(\frac{8\ell}{\bd}\Bigr)^{s} .
\end{equation}
Taking $2s$-th roots gives \eqref{eq:even-spectral}.
\end{proof}

\begin{proof}[Proof of \cref{thm:even-main}]
We may assume $\cH\ne\varnothing$, so that $\bd>0$.  Suppose $\cH$ has
no even cover of size at most $5\ell\log(en/\ell)$.  Since $\ell\ge2k$,
$k\ge4$ and $\log(en/\ell)\ge1$,
\[
        \max\{2s,\ \ell+ks+1\}
        \le\ell+3\ell\log(en/\ell)+k+1
        \le5\ell\log(en/\ell),
\]
so \cref{lem:even-indegree} applies.  Combining its upper bound
\eqref{eq:even-spectral} with the lower bound $\norm K\ge\frac12$ gives
$\bd\le32\,\ell\,N^{1/s}\le32\,\ell\,e^{k/3}$ by
\eqref{eq:even-moment}.  Comparing with \cref{lem:even-density},
\[
        |\cH|\le\frac{640}9\,\sqrt r\,
        \Bigl(\frac{e^{5/3}}8\Bigr)^{r}
        \ell\Bigl(\frac n\ell\Bigr)^{r} ,
\]
and $\sqrt r\,(e^{5/3}/8)^{r}$ decreases in $r\ge2$, since
$\tfrac1{2r}\le\tfrac14<\log(8/e^{5/3})$.  Its value at $r=2$ is
$\sqrt2\,(e^{5/3}/8)^2\le\tfrac58$, so
\[
        |\cH|\le\frac{400}9\,\ell\Bigl(\frac n\ell\Bigr)^{k/2}
        \le45\,\frac{n^{k/2}}{\ell^{k/2-1}}. \qedhere
\]
\end{proof}

\section{Odd uniformity}\label{sec:odd}

Let $k=2r+1\ge3$ be odd.  We keep the central range \eqref{eq:central}
in force, with three exceptions: \cref{def:spread}, the packing bound of \cref{lem:odd-remainder}, the pigeonhole \eqref{eq:dominant}, and \cref{prop:odd-trace} use neither \eqref{eq:central} nor the parity of $k$, and are reused for all levels and both parities in
\cref{sec:completion}.

\begin{theorem}\label{thm:odd-main}
Assume \eqref{eq:central}.  There are absolute constants $A$ and $C$
such that, if $\cH$ has no even cover of size at most
$A\,\ell\log(en/\ell)$, then
\[
        |\cH|\le C\,\frac{n^{k/2}}{\ell^{k/2-1}}.
\]
\end{theorem}

\subsection{Spread bundles and a dominant level}

\begin{definition}[Spread bundle]\label{def:spread}
For $1\le q\le k-1$, put
$h_q:=\max\{2,\lceil(n/\ell)^{\,q-k/2}\rceil\}$.
A \emph{spread bundle of residual size} $q$ is a family $\cB\subseteq\cH$
of $h_q$ hyperedges through a common $(k-q)$-set $U_\cB$, so that every
$F\in\cB$ decomposes as $F=U_\cB\mathbin{\dot\cup}R_F$ with
\emph{residual} $R_F$ of size $q$, such that
\begin{equation}\label{eq:spread}
        \#\{J\in\cB: T\subseteq R_J\}<h_{q-t}
\end{equation}
for every $F\in\cB$ and every $T\subseteq R_F$ with
$1\le t:=|T|<q$.
\end{definition}

Choose a maximal hyperedge-disjoint collection of spread bundles,
simultaneously over all residual sizes.  Let $m_q$ be the number of
hyperedges in bundles of residual size $q$, put $m_+:=\sum_qm_q$, and
let $m_0:=|\cH|-m_+$ be the number of remaining hyperedges.

\begin{lemma}\label{lem:odd-remainder}
For all $1\le\ell\le n$, the packing swallows all but a sparse
remainder:
\[
        m_0\le\frac2k\,\frac{n^{k/2}}{\ell^{k/2-1}}.
\]
\end{lemma}

\begin{proof}
Call a hyperedge \emph{free} if it lies in no chosen bundle, and for
$U\subseteq[n]$ with $1\le|U|\le k-1$ let $d_0(U)$ be the number of free
hyperedges containing $U$.  Suppose some vertex $v$ had
$d_0(\{v\})\ge h_{k-1}$.  Among all $U\ni v$ with
$d_0(U)\ge h_{k-|U|}$, choose one whose residual size $q:=k-|U|$ is
smallest, and let $\cB$ consist of any $h_q$ free hyperedges containing
$U$.  Then $\cB$ is a spread bundle with center $U$: a violation of
\eqref{eq:spread} at some $T$ of size $t$ would give
$d_0(U\cup T)\ge h_{q-t}$ with residual size $q-t<q$, contradicting
minimality.  Adding $\cB$ contradicts the maximality of the packing.

Hence $d_0(\{v\})<h_{k-1}\le2(n/\ell)^{k/2-1}$ for every vertex $v$,
using $n\ge\ell$.  Summing over vertices,
$km_0<2n(n/\ell)^{k/2-1}=2n^{k/2}/\ell^{k/2-1}$.
\end{proof}

\begin{lemma}\label{lem:odd-pairs}
In the central range, every spread bundle of residual size $q$ contains
at least $h_q^2/4$ unordered pairs $\{F,J\}$ with
$R_F\cap R_J=\varnothing$.
\end{lemma}

\begin{proof}
Let $\cB$ be the bundle.  If $q\le r+1$, then $h_{q-1}=2$ when $q\ge2$,
so no vertex lies in two residuals of $\cB$: a common vertex $v$ would
put two members through $T=\{v\}$, violating \eqref{eq:spread}; for
$q=1$ the residuals are distinct singletons because $\cH$ is simple.
The residuals are thus pairwise disjoint and all
$\binom{h_q}2\ge h_q^2/4$ pairs qualify (for $q\le r$ one has
$(n/\ell)^{q-k/2}\le1$, so $h_q=2$ and $\cB$ is a single pair).

If $q\ge r+2$, fix $F\in\cB$.  For each $v\in R_F$, \eqref{eq:spread}
bounds the number of members other than $F$ whose residual contains $v$
by $h_{q-1}-2$.  Since $q-1>k/2$,
$h_{q-1}\le(n/\ell)^{q-1-k/2}+1\le2(n/\ell)^{q-1-k/2}$, so the number of
members meeting $R_F$ is at most
\[
        2q\,(n/\ell)^{\,q-1-k/2}
        \le\frac{2k\ell}{n}\,(n/\ell)^{\,q-k/2}
        \le\frac{h_q}{50}
        \le\frac{h_q}4 .
\]
Hence $F$ has at least $h_q-1-h_q/4\ge h_q/2$ residual-disjoint
partners, using $h_q\ge(n/\ell)^{3/2}\ge4$; the number of unordered
disjoint pairs is at least $h_q\cdot(h_q/2)/2=h_q^2/4$.
\end{proof}

If $m_+=0$, then $|\cH|=m_0\le n^{k/2}/\ell^{k/2-1}$ and
\cref{thm:odd-main} holds.  Otherwise fix once and for all a residual
size $q$ with
\begin{equation}\label{eq:dominant}
        m_q\ge\frac{m_+}{k-1};
\end{equation}
only the bundles of residual size $q$ are used from now on.

\subsection{Balanced transitions and exact excess weights}

Put $a:=\lfloor q/2\rfloor$, $b:=\lceil q/2\rceil$, and
$B_q:=\binom qa$.  For every residual-disjoint pair $F,J$ in one bundle of residual size
$q$, retain every \emph{balanced transition}
\begin{equation}\label{eq:balanced}
 S\longleftrightarrow S':=S\sd R_F\sd R_J,
 \qquad
 \{|S\cap R_F|,\,|S\cap R_J|\}
 =\{a,b\},
\end{equation}
recorded together with its label pair $\{F,J\}$; both rows lie in
$\Om_\ell$, because $a+b=q$.  Choosing
$S\cap R_F$, $S\cap R_J$, and the remaining $\ell-q$ vertices of $S$
freely, and noting that each instance arises from at most two such
choices, a fixed pair creates at least
$\tfrac12B_q^2\binom{n-2q}{\ell-q}$ instances.  By
\cref{lem:odd-pairs}, the total number $\mu_q$ of instances satisfies
\begin{equation}\label{eq:odd-mass}
        \mu_q\ge\frac18\,m_qh_qB_q^2\binom{n-2q}{\ell-q}.
\end{equation}

A \emph{port} is a row--hyperedge pair $z=(S,F)$, the point of contact through which a move at $S$ may use $F$.  If an instance
$e$ has endpoint rows $S,S'$ and labels $F,J$, write
$P(e):=\{(S,F),(S,J),(S',F),(S',J)\}$ for its four ports, and let $c(z)$ be the number of instances whose
port set contains $z$.  Define
\begin{equation}\label{eq:odd-weights}
        \delta_e:=1+\sum_{z\in P(e)}\bigl(c(z)-1\bigr),
        \qquad w_e:=\delta_e^{-1}.
\end{equation}
This exact form yields both of the properties we need by identities.
Since $\delta_e\ge c(z)$ for each $z\in P(e)$, every port has capacity
one:
\begin{equation}\label{eq:odd-port-capacity}
        \sum_{e:\,z\in P(e)}w_e\le1.
\end{equation}
Since each port $z$ lies in exactly $c(z)$ instances,
$\sum_e\delta_e=\mu_q+\sum_zc(z)(c(z)-1)$, so Cauchy--Schwarz, in the
form $\sum_ew_e\ge \mu_q^2/\sum_e\delta_e$, gives
\begin{equation}\label{eq:retained-mass}
        \sum_ew_e
        \ \ge\ \frac{\mu_q^2}{\mu_q+\sum_zc(z)(c(z)-1)}.
\end{equation}

The next lemma is the only counting estimate the odd case requires.
Its exponent $b-k/2$ is negative because $k$ is odd: $q\le k-1$ forces
$b=\lceil q/2\rceil\le(k-1)/2$, so $b-k/2\le-\tfrac12$; equivalently,
the balanced half $a=\lfloor q/2\rfloor$ always exceeds $q-k/2$ by at
least one half.

\begin{lemma}\label{lem:odd-contention}
If $q\le r$, then $c(z)\le1$ for every port $z$.  If $q\ge r+1$, then
the total excess contention obeys
\begin{equation}\label{eq:odd-contention}
        \sum_zc(z)\bigl(c(z)-1\bigr)\le\Xi_q\,\mu_q,
        \qquad
        \Xi_q:=8\,B_q\,(2.031)^{a}\,\Bigl(\frac n\ell\Bigr)^{b-k/2}.
\end{equation}
\end{lemma}

\begin{proof}
If $q\le r$, every bundle is a single residual-disjoint pair $\{F,J\}$,
and a port $(S,F)$ determines the partner hyperedge $J$ and thereby the second
row $S\sd R_F\sd R_J$: at most one instance passes through it.

Let $q\ge r+1$.  Since each port $z$ lies in exactly $c(z)$ instances,
$\sum_zc(z)(c(z)-1)=\sum_e\sum_{z\in P(e)}(c(z)-1)$, so it suffices to
show that, for every pair $F,J$, the average of
$\sum_{z\in P(e)}(c(z)-1)$ over the instances $e$ of that pair is at
most $\Xi_q$.  By symmetry we bound the contribution of the port
$(S,F)$ and multiply by four.

Fix an instance $e$ with rows $S,S'$ and labels $F,J$. Every packed hyperedge lies in a unique selected bundle; for such an $F$,
write $\cB(F)$ for that bundle. An instance
$e'\ne e$ with $(S,F)\in P(e')$ has a label pair $\{F,J'\}$ with
$J'\in\cB(F)$: its second row is $S\sd R_F\sd R_{J'}$, so $e'$ is
determined by $J'$, and $J'=J$ returns $e'=e$.  Thus
$c((S,F))-1$ is at most the number of $J'\in\cB(F)\setminus\{F,J\}$
with $R_{J'}\cap R_F=\varnothing$ and
$|S\cap R_{J'}|=q-|S\cap R_F|$.

Let $t:=|R_{J'}\cap R_J|$.  If $t\ge1$, then $J$ and $J'$ are two bundle
members whose residuals contain $T:=R_{J'}\cap R_J$, so
\eqref{eq:spread} forces $h_{q-t}>2$, hence
$(n/\ell)^{q-t-k/2}>1$ and $t\le q-\tfrac{k+1}2$; in particular $t<a$,
since $a-\bigl(q-\tfrac{k+1}2\bigr)=\tfrac{k+1}2-b\ge1$.  For a fixed
trace $T\subseteq R_J$ of size $t\ge1$, the spread condition bounds
the number of admissible $J'$ by
$h_{q-t}\le2(n/\ell)^{q-t-k/2}$; for $t=0$ the trivial bound
$h_q\le2(n/\ell)^{q-k/2}$ is of the same form.

Now average over the instances of $(F,J)$ with the patterns
$S\cap R_F$ and $S\cap R_J$ fixed; the remaining $\ell-q$ vertices of
$S$ form a uniformly random $(\ell-q)$-subset of the $n-2q$ vertices
off $R_F\cup R_J$.  For fixed $J'$ with trace size $t$, the event
$|S\cap R_{J'}|=q-|S\cap R_F|\ge a$ requires at least $a-t$ of the
$q-t$ vertices of $R_{J'}\setminus R_J$ to be chosen, so its
probability is at most
\[
        \binom{q-t}{a-t}
        \left(\frac{\ell-q}{n-3q}\right)^{a-t},
\]
using $(\ell-q)_j/(n-2q)_j\le((\ell-q)/(n-3q))^j$ for $j\le q$.  The
quantity $y:=n(\ell-q)/(\ell(n-3q))$ satisfies $y\le n/(n-3q)\le1.031$,
because $3q\le3k\le3n/100$ by \eqref{eq:central}.  Summing over
the $\binom qt$ traces and the admissible $t$, and using
$\binom qt\binom{q-t}{a-t}=B_q\binom at$, the expected excess
contention at the port $(S,F)$ is at most
\[
        \sum_{t\ge0}\binom qt\cdot2\Bigl(\frac n\ell\Bigr)^{q-t-k/2}
        \binom{q-t}{a-t}\Bigl(\frac{y\ell}n\Bigr)^{a-t}
        =2B_q\Bigl(\frac n\ell\Bigr)^{b-k/2}\sum_{t\ge0}\binom at y^{\,a-t}
        =2B_q\Bigl(\frac n\ell\Bigr)^{b-k/2}(1+y)^a,
\]
where we used $(q-t-k/2)-(a-t)=b-k/2$, which is independent of $t$.
Finally
$(1+y)^a\le(2.031)^a$, and multiplying by the four ports gives
\eqref{eq:odd-contention}.
\end{proof}

Let $Y$ be the weighted labeled multigraph consisting of all instances
at the dominant level $q$ with the weights \eqref{eq:odd-weights}, and
let $\bd:=N^{-1}\sum_Sd(S)=2N^{-1}\sum_ew_e$ be its weighted average
degree, where $N:=\binom n\ell$ as before; note $\bd>0$, since
$m_q>0$ and $n\ge\ell+q$ make $\mu_q>0$ in \eqref{eq:odd-mass}.

\begin{lemma}\label{lem:odd-gain}
With the weights \eqref{eq:odd-weights}, the graph $Y$ converts
hyperedges into degree according to
\begin{equation}\label{eq:odd-gain}
        \bd\ \ge\ \frac{(1.0324)^k}{74\,k^3}\,
        m_+\Bigl(\frac\ell n\Bigr)^{k/2}.
\end{equation}
\end{lemma}

\begin{proof}
Write $\Sigma:=\sum_zc(z)(c(z)-1)$.  By \eqref{eq:odd-mass},
\cref{lem:ratio}, and $B_q\ge2^q/(q+1)$,
\begin{equation}\label{eq:odd-mass-converted}
 \frac{\mu_q}N
 \ \ge\ \frac9{80}\,m_q\,h_qB_q^2
 \Bigl(\frac2e\Bigr)^{q}\Bigl(\frac\ell n\Bigr)^{q}
 \ \ge\ \frac9{80}\,\frac{m_q}{(q+1)^2}
 \Bigl(\frac8e\Bigr)^{q}h_q\Bigl(\frac\ell n\Bigr)^{q},
\end{equation}
and in either regime of $h_q$,
\begin{equation}\label{eq:h-unified}
 \Bigl(\frac8e\Bigr)^{q}h_q\Bigl(\frac\ell n\Bigr)^{q}
 \ \ge\ \Bigl(\frac8e\Bigr)^{k/2}\Bigl(\frac\ell n\Bigr)^{k/2}:
\end{equation}
for $q\le k/2$ use $h_q\ge1$ and $n/\ell\ge8/e$, and for $q>k/2$ use
$h_q\ge(n/\ell)^{q-k/2}$.

If $\Sigma\le \mu_q$ (in particular whenever $q\le r$, where
$c(z)\le1$), then \eqref{eq:retained-mass} gives
$\sum_ew_e\ge \mu_q/2$, so
\[
 \bd\ \ge\ \frac{\mu_q}N
 \ \ge\ \frac9{80}\,\frac{m_q}{(q+1)^2}
 \Bigl(\frac8e\Bigr)^{k/2}\Bigl(\frac\ell n\Bigr)^{k/2}.
\]
If $\Sigma>\mu_q$, then $q\ge r+1$ and
\eqref{eq:retained-mass} and \cref{lem:odd-contention} gives $\sum_ew_e\ge \mu_q/(2\Xi_q)$.  Using the
first bound in \eqref{eq:odd-mass-converted}, together with
$h_q\ge(n/\ell)^{q-k/2}$, $B_q\ge2^q/(q+1)$, and
$(2.031)^{-a}\ge(2.031)^{-q/2}$, and cancelling one $B_q$,
\[
 \bd\ \ge\ \frac{\mu_q}{N\,\Xi_q}
 \ \ge\ \frac9{640}\,\frac{m_q}{q+1}\,
 \gamma^{\,q}\Bigl(\frac n\ell\Bigr)^{k/2-b}
 \Bigl(\frac\ell n\Bigr)^{k/2},
 \qquad
 \gamma:=\frac4{e\sqrt{2.031}}\ge1.0324 .
\]
Since $2b\le q+1$ and $n/\ell\ge100k>\gamma^2$, the map
$q\mapsto\gamma^{q}(n/\ell)^{(k-q-1)/2}$ is decreasing, so
\[
 \gamma^{\,q}\Bigl(\frac n\ell\Bigr)^{k/2-b}
 \ \ge\ \gamma^{\,q}\Bigl(\frac n\ell\Bigr)^{(k-q-1)/2}
 \ \ge\ \gamma^{\,k-1}
 \ \ge\ \frac{(1.0324)^k}{1.0324} ,
\]
using $q\le k-1$ at the second step; and $9/(640\cdot1.0324)\ge1/74$.

In both cases $\bd\ge\frac1{74}(q+1)^{-2}(1.0324)^km_q(\ell/n)^{k/2}$;
now \eqref{eq:dominant} and $(q+1)^2\le k^2$ give \eqref{eq:odd-gain}.
\end{proof}

\subsection{The paired trace estimate}\label{sec:paired}

\begin{proposition}\label{prop:odd-trace}
Let $s\ge1$.  If $\cH$ has no even cover of size
at most $\max\{4s,\ 2\ell+4ks+2\}$, then $\bd\le32\,\ell\,N^{1/s}$.
\end{proposition}

\begin{proof}
We run the argument of \cref{sec:even} on $Y$ and record what
changes: each step now toggles \emph{two} labels, and the bundle
centers, which the moves never see, are restored by one formal tag per
bundle.

Let $A_Y$ be the weighted adjacency matrix, and put
$\Gamma:=\diag(d(S))+\bd I$ and $K:=\Gamma^{-1/2}A_Y\Gamma^{-1/2}$;
testing on $\Gamma^{1/2}\1$ gives $\norm K\ge\frac12$.  The lift has
states $(S,M)$ with $|M|\le2s$; an edge of $Y$ with labels $F,J$ sends
$(S,M)\leftrightarrow(S\sd R_F\sd R_J,\,M\sd\{F,J\})$ with weight
$w_e(\Gamma(S)\Gamma(S\sd R_F\sd R_J))^{-1/2}$, and $\Lop$ is the
restriction to the components of the empty-memory states.

\emph{Trace comparison.}  As in \cref{lem:even-trace},
$\Tr K^{2s}=\sum_S\bigl\langle(S,\varnothing)\bigr|\Lop^{2s}
\bigl|(S,\varnothing)\bigr\rangle$, both sides counting the same closed
walks with the same weights: a closed $2s$-step walk of $Y$ determines,
and is determined by, its sequence of edges, so it lifts to a unique
walk started at $(S,\varnothing)$, and conversely a closed lifted walk
at $(S,\varnothing)$ projects to one.  For a closed row
walk with odd-label set $Q$, row closure gives
$\bigsd_{F\in Q}R_F=\varnothing$; each step uses two labels from one
bundle, so every bundle meets $Q$ evenly and the centers cancel in
pairs, $\bigsd_{F\in Q}F=\varnothing$.  A nonempty $Q$ would thus be
an even cover of size at most $4s$, which is forbidden; hence every
closed row walk returns to empty memory, with memory sizes at most
$\min\{2j,2(2s-j)\}\le2s$ along the way, so the truncation discards
nothing.

\emph{Tagged representatives.}  For a lifted state put
$W(S,M):=S\cup\bigcup_{F\in M}R_F$, so $|W(S,M)|\le\ell+2ks$.
Introduce one formal coordinate $\tau_\cB$ per bundle, and associate to
each packed hyperedge $F$ the augmented vector
$v_F:=\1_{R_F}+\1_{\{\tau_{\cB(F)}\}}$ and the augmented set
$\widehat F:=R_F\cup\{\tau_{\cB(F)}\}$.  For every $W\subseteq[n]$
with $|W|\le\ell+2ks$, the family $\{v_F:R_F\subseteq W\}$ is linearly
independent: a minimal dependency $\cC$, say of size $p$ and meeting
$\beta$ bundles, meets every bundle evenly because of the tag
coordinates, so $\beta\le p/2$, while minimality gives
$p-1\le|W|+\beta$, whence
$p\le2|W|+2\le2\ell+4ks+2$; summing the members of $\cC$, the
residuals cancel and the even bundle counts cancel the centers, so
$\cC$ is a forbidden even cover.  As in \cref{lem:even-window},
independence implies Hall's condition, and the sets $\widehat F$ with
$R_F\subseteq W$ admit a system of distinct representatives $\rho_W$.

\emph{Orientation.}  Orient by the rules of \cref{sec:even}, with one
change in the choice of witness: for an equal-support edge with labels
$F,J$, the sets $\widehat F,\widehat J$ share the bundle tag, so at
least one label---after a deterministic tie-break---has a genuine
representative $v\in[n]$, and since $R_F\cup R_J=S_x\sd S_y$, this
witness lies in exactly one endpoint row; orient toward it.  For
unequal supports, orient toward an endpoint whose support contains a
vertex absent from the other, the witness being the smallest such
vertex; here the two supports need no longer be nested, because a step
may enter one label and remove the other, so we break ties
deterministically.

\emph{Incoming weight at most $2\ell$.}  For equal-support in-edges at
$x=(S,M)$, the witness $v\in S$ determines the label
$F_v:=\rho_{W(x)}^{-1}(v)$, so all in-edges with witness $v$ project
to distinct edges of $Y$ at the port $(S,F_v)$, of total weight at
most $1$ by \eqref{eq:odd-port-capacity}; and there are at most $\ell$
witnesses.  For unequal-support in-edges, the argument of
\cref{lem:even-indegree} applies with residuals in place of
hyperedges.  Its first two steps are unchanged: the witness
$v\in W(x)\setminus W(y)$ lies in $S$, because the step moves exactly
the vertices of $R_F\cup R_J$; it belongs to the residual of exactly
one of the two toggled labels, their residuals being disjoint; and
that label was entered, since a removed label keeps its residual
inside $W(y)$.  The one new point is uniqueness within $M$: a second
member of $M$ whose residual contains $v$ either lies in $M_y$,
putting $v\in W(y)$, or is the step's other label---then also entered,
with residual disjoint from the first.  Hence $F_v$ depends only on
$(x,v)$, all in-edges with witness $v$ use the port $(S,F_v)$, and the
incoming weight is at most $2\ell$ in total.

\emph{Conclusion.}  The weighted case of \cref{lem:orientation} gives
$\norm\Lop\le2\sqrt{2\ell/\bd}$, and the chain \eqref{eq:even-chain}
holds verbatim, giving $\norm K\le2\sqrt{2\ell/\bd}\,N^{1/(2s)}$;
combined with $\norm K\ge\frac12$ this rearranges to
$\bd\le32\,\ell\,N^{1/s}$.
\end{proof}

\begin{proof}[Proof of \cref{thm:odd-main}]
We prove this with $A=165$.  Set
$s:=\lceil\frac{40\,\ell}k\log\frac{en}\ell\rceil$, so that
$N^{1/s}\le e^{k/40}$ and, using $\ell\ge2k$, $\log(en/\ell)\ge1$, and
$2\ell+4k+2\le5\ell$,
\[
        \max\{4s,\ 2\ell+4ks+2\}
        \le2\ell+160\,\ell\log(en/\ell)+4k+2
        \le165\,\ell\log(en/\ell).
\]
If $\cH$ has no even cover of size at most $165\,\ell\log(en/\ell)$,
then \cref{prop:odd-trace} gives $\bd\le32\,\ell\,e^{k/40}$, and
comparing with \cref{lem:odd-gain},
\[
        m_+\le2368\,k^3
        \Bigl(\frac{e^{1/40}}{1.0324}\Bigr)^{k}
        \,\ell\Bigl(\frac n\ell\Bigr)^{k/2}
        \le C\,\frac{n^{k/2}}{\ell^{k/2-1}},
\]
since $k^3(e^{1/40}/1.0324)^k\le k^3(0.994)^k$ is bounded.  Adding the
remainder bound of \cref{lem:odd-remainder} gives
$|\cH|\le C\,n^{k/2}/\ell^{k/2-1}$.
\end{proof}

\section{All levels}\label{sec:completion}

We now remove the central-range hypothesis \eqref{eq:central}.

\subsection{Above the central range}

The dense regime is elementary: it follows from counting syndromes.

\begin{lemma}\label{lem:dense}
Suppose $n\le100\,k\ell$.  Then at least one of
\[
        g_{\rm ev}(\cH)\le C\ell\log\frac{en}\ell,
        \qquad
        |\cH|\le C\,\frac{n^{k/2}}{\ell^{k/2-1}}
\]
holds.
\end{lemma}

\begin{proof}
Put $t:=\lceil400\,\ell\log(en/\ell)\rceil$ and $m:=|\cH|$, and
suppose that $g_{\rm ev}(\cH)>2t$; since $2t\le C\ell\log(en/\ell)$, it
suffices to show $m\le C\ell(n/\ell)^{k/2}$.  We will use twice the
elementary bound $\log(en/\ell)\le2(n/\ell)^{1/2}$.  If $m<t$, then
$m\le C\ell\log(en/\ell)\le C\ell(n/\ell)^{1/2}$ and we are done.  If
$m\ge t$, consider the map sending a $t$-element subfamily
$\cT\subseteq\cH$ to $\sum_{F\in\cT}\1_F\in\F_2^n$.  It is injective,
because two subfamilies with the same image differ by a nonempty even
cover of size at most $2t$; hence $(m/t)^t\le\binom mt\le2^n$ and
$m\le t\,2^{n/t}\le
t\exp\bigl((\log2)(n/\ell)/(400\log(en/\ell))\bigr)$.
If $n/\ell\le2$, the exponent is at most $1$ and $m\le et$ and we conclude as in the case $m < t$.  If
$n/\ell\ge2$, then $n/\ell\le100\,k$, $\log(en/\ell)\ge1$, and
$\log(n/\ell)\ge\log2$ bound the exponent by
$(\log2)\,k/4\le(k/4)\log(n/\ell)$, so $m\le t\,(n/\ell)^{k/4}
\le C\ell\,(n/\ell)^{1/2+k/4}\le C\ell(n/\ell)^{k/2}$.
\end{proof}

\subsection{Crossed halves}

We now deal with the remaining regime $k>400\,\ell$. 

\begin{proposition}\label{prop:half}
Let $k>400\,\ell$, of either parity.  Then at least one of
\[
        g_{\rm ev}(\cH)\le22\,\ell\log\frac{en}\ell,
        \qquad
        |\cH|\le50\,\frac{n^{k/2}}{\ell^{k/2-1}}
\]
holds.
\end{proposition}

\begin{proof}
We may assume $n\ge k$, since
otherwise $\cH=\varnothing$.  Run the spread-bundle packing of
\cref{sec:odd}, for either parity of $k$: by \cref{lem:odd-remainder},
$m_0\le\frac2k\,n^{k/2}/\ell^{k/2-1}$.  If $m_+=0$ we are done, so fix
a dominant residual size $q$ with $m_q\ge m_+/(k-1)$.

Put $a:=\lfloor q/2\rfloor$ and $b:=\lceil q/2\rceil$, and split every
residual at the dominant level canonically,
$R_F=R_F^-\mathbin{\dot\cup}R_F^+$, where $R_F^-$ consists of the $a$
smallest elements.  For $q\ge2$, as coordinates we take the
\emph{signed halves}
$\bigl(\binom{[n]}a\times\{-\}\bigr)
\sqcup\bigl(\binom{[n]}b\times\{+\}\bigr)$; the rows are the pairs
\[
 \Om:=\Bigl\{\{(D,-),(E,+)\}:
 D\in\binom{[n]}a,\ E\in\binom{[n]}b\Bigr\},
 \qquad
 N_q:=|\Om|=\binom na\binom nb,
\]
of size two, and the \emph{moving set} of a packed hyperedge $F$ is
the two-element set $\Rh F:=\{(R_F^-,-),(R_F^+,+)\}$.  For $q=1$,
where the halves degenerate, the rows are instead the singletons of
$[n]$, the moving set of $F$ is $R_F$ itself, and $N_1:=n$; every
step below simplifies accordingly.

Within one bundle, the map $F\mapsto R_F^-$ is injective when $q\ge2$:
two members whose residuals share a set $T$ of size $a$ would violate
\eqref{eq:spread}, because $b\le k/2$ gives
$(n/\ell)^{\,b-k/2}\le1$ and hence $h_{q-a}=h_b=2$.  Likewise
$F\mapsto R_F^+$ is injective, as $h_a=2$.  Hence
$\Rh F\cap\Rh J=\varnothing$ for any two distinct members $F,J$ of one
bundle; for $q=1$ this holds because the residuals are distinct
singletons, $\cH$ being simple.  The transition graph has, for every
unordered pair $F\ne J$ in one dominant-level bundle, the single edge
\begin{equation}\label{eq:crossed}
 \{(R_F^-,-),(R_J^+,+)\}
 \ \longleftrightarrow\
 \{(R_J^-,-),(R_F^+,+)\},
\end{equation}
with labels $F,J$ and weight $1$; both rows lie in $\Om$, they are
distinct by half-injectivity, and \eqref{eq:crossed} toggles exactly
the two moving sets.  Every port has capacity one: fix a port $(S,F)$
with $S=\{(D,-),(E,+)\}$.  An edge at $S$ with label $F$ and partner
$J$ satisfies either $D=R_F^-$ and $R_J^+=E$, or $E=R_F^+$ and
$R_J^-=D$; if both $D=R_F^-$ and $E=R_F^+$ hold, either case forces
$J=F$, so no edge uses the port, and otherwise at most one case
applies and determines $J$ by half-injectivity.  (For $q=1$, the
label $F$ has a unique residual-disjoint partner in its bundle, so
again at most one edge uses each port.)

For the density, every bundle contributes all $\binom{h_q}2\ge
h_q^2/4$ of its pairs, so the graph has at least $m_qh_q/4$ edges,
and $N_q\le n^q/(a!\,b!)$ gives
\[
        \bd\ \ge\ \frac{m_qh_q}{2N_q}
        \ \ge\ \frac{m_q\,h_q\,a!\,b!}{2\,n^q}
        \ \ge\ \frac{m_q}2\Bigl(\frac k{4en}\Bigr)^{k/2},
\]
where the last step is the claim
\begin{equation}\label{eq:factorial-gain}
        h_q\,a!\,b!\ \ge\ n^q\Bigl(\frac k{4en}\Bigr)^{k/2}
        \qquad(1\le q\le k-1).
\end{equation}
For $q=1$ this reads $2\ge n(k/4en)^{k/2}$, which holds because
$n\ge k$ gives $n(k/4en)^{k/2}\le k\,(4e)^{-k/2}\le1$.  For $q\ge2$,
$a!\,b!=q!/\binom qa\ge(q/e)^q2^{-q}=(q/2e)^q$.  If $q>k/2$, then
$h_q\ge(n/\ell)^{q-k/2}$ and $q/(2e\ell)\ge k/(4e\ell)>1$, so
$h_q\,a!\,b!\,n^{-q}\ge(q/(2e\ell))^{q}(\ell/n)^{k/2}
\ge(k/(4e\ell))^{k/2}(\ell/n)^{k/2}=(k/(4en))^{k/2}$.  If $q\le k/2$,
then, using $n\ge k$ and the fact that $t\mapsto(t/(2ek))^t$ decreases
for $t\le2k$,
\[
        \frac{a!\,b!\,n^{-q}}{(k/4en)^{k/2}}
        \ge\Bigl(\frac q{2e}\Bigr)^{q}k^{k/2-q}
        \Bigl(\frac{4e}k\Bigr)^{k/2}
        =\Bigl(\frac q{2ek}\Bigr)^{q}(4e)^{k/2}
        \ge\Bigl(\frac1{4e}\Bigr)^{k/2}(4e)^{k/2}=1 .
\]

The proof of \cref{prop:odd-trace} now applies verbatim, with the
signed halves in place of the vertices of $[n]$ and the moving sets in
place of the residuals (for $q=1$ it applies literally).  Indeed, it
used only that each edge of the graph toggles the disjoint moving
sets of its two labels together with the memory, and that every port
has capacity one---except in its trace-comparison and
tagged-representative steps, where the residuals of the odd labels
cancel; there one first applies the linear map
$\1_{(D,\pm)}\mapsto\1_D$, under which the moving sets again sum to
the residuals, and the tags restore the centers as before.  Rows now
have size at most $2$, so the incoming weight is at most $4$, and
windows have size at most $2+4s$, so the required cover radius is at
most $\max\{4s,\,2(2+4s)+2\}=8s+6$.

Set $s:=\lceil\ell\log(en/\ell)\rceil$.  Then
$\log N_q\le q\log(en)\le q\,\ell\log(en/\ell)
\le k\,\ell\log(en/\ell)$, using
$\log(en)=\log(en/\ell)+\log\ell\le\ell\log(en/\ell)$, so
$N_q^{1/s}\le e^{k}$, while
$8s+6\le8\ell\log(en/\ell)+14\le22\,\ell\log(en/\ell)$.  Thus, if
$\cH$ has no even cover of size at most $22\,\ell\log(en/\ell)$, the
chain \eqref{eq:even-chain}, with the incoming-weight bound $4$ in
place of $2\ell$, gives $\norm K\le2\sqrt{4/\bd}\,N_q^{1/(2s)}$ and
hence $\bd\le64\,N_q^{1/s}\le64\,e^k$.  Comparing with the density
bound, and using $4e^3\ell/k\le4e^3/400\le\tfrac14$,
\[
        m_q\le128\,e^k\Bigl(\frac{4en}k\Bigr)^{k/2}
        =\frac{128}\ell\Bigl(\frac{4e^3\ell}k\Bigr)^{k/2}
        \frac{n^{k/2}}{\ell^{k/2-1}}
        \le\frac{128}{2^{k}}\,
        \frac{n^{k/2}}{\ell^{k/2-1}},
\]
so $m_+\le(k-1)\,m_q\le128\,k\,2^{-k}\,n^{k/2}/\ell^{k/2-1}
\le48\,n^{k/2}/\ell^{k/2-1}$.  Adding $m_0$ completes the proof.
\end{proof}

\subsection{Proof of \texorpdfstring{\cref{thm:main}}{the main theorem}}

\begin{proof}[Proof of \cref{thm:main}]
If $n\le100\,k\ell$, apply \cref{lem:dense}.

Suppose next that $n>100\,k\ell$ and $k\le400\,\ell$, and put
$\lambda:=\max\{\ell,2k\}$, so that $\ell\le\lambda\le800\,\ell$ and
$\lambda\le n$, because $n\ge100\,k\ell\ge\max\{\ell,2k\}$.  If
$n\ge100\,k\lambda$, the central range \eqref{eq:central} holds at the
level $\lambda$ and \cref{thm:even-main} or \cref{thm:odd-main}
applies there; otherwise \cref{lem:dense} applies at the level
$\lambda$.  In either case, at least one of
\[
        g_{\rm ev}(\cH)\le C\lambda\log\frac{en}\lambda
        \le800\,C\,\ell\log\frac{en}\ell,
        \qquad
        |\cH|\le C\,\frac{n^{k/2}}{\lambda^{k/2-1}}
        \le C\,\frac{n^{k/2}}{\ell^{k/2-1}}
\]
holds, using $\log(en/\lambda)\le\log(en/\ell)$, $\lambda\ge\ell$, and
$k\ge3$.

Finally, if $k>400\,\ell$, apply \cref{prop:half}; when $n<k$ the
hypergraph is empty and there is nothing to prove.  All constants used are absolute and independent of $k$.
\end{proof}

\end{document}